\documentclass[11pt]{article}
 \usepackage{amsopn}
 \usepackage{amsmath,amsthm,amssymb}
 
\usepackage[latin1]{inputenc}
\usepackage{multicol}\usepackage[spanish,activeacute]{babel}
\usepackage{graphicx}
\usepackage{amsmath}
\usepackage{amsfonts}
\newtheorem{lm}{Lema}

\usepackage[latin1]{inputenc}

\newtheorem{thm}[lm]{Theorem}
\newtheorem{lem}[lm]{Lemma}

\newtheorem{cor}[lm]{Corollary}

\newtheorem{defin}[lm]{Definition}

\newcommand{\RR}{\mathbb{R}}

\newcommand{\Exp}{\text{Exp }}

\begin{document}

\title{Nonlocal diffusion equations in Carnot groups}

\author{I. E. Cardoso\footnote{Fac. de Cs. Exactas, Ingenier\'ia y Agrimensura, Pellegrini 250, Rosario, Santa Fe, Argentina, $isolda@fceia.unr.edu.ar$.}
            - 
           R. E. Vidal\footnote{Fac. de Matem\'atica, Astronom\'ia y F\'isica, Medina Allende s/n, C\'ordoba, C\'ordoba, Argentina, $vidal@mate.uncor.edu.ar$.}
}

\date{}

\maketitle

\begin{abstract}
Let $G$ be a Carnot group. We study nonlocal diffusion equations in a domain $\Omega$ of $G$ of the form 
$$
u_t^\epsilon(x,t)=\int_{G}\frac{1}{\epsilon^2}K_{\epsilon}(x,y)(u^\epsilon(y,t)-u^\epsilon(x,t))\,dy, \qquad x\in \Omega
$$
with $u^\epsilon=g(x,t)$ for $x\notin\Omega$. For an appropriated rescaled kernel $K_\epsilon$ we prove that solutions $u^\epsilon$, when $\epsilon\rightarrow0$,  uniformly approximate the solution of different local Dirichlet problem in $G$. The key tool used is the Taylor series development for a function defined on a Carnot group.
\end{abstract}

\section{Introduction}\label{sec:Intro}
The nonlocal diffusion problems in the Euclidean space $\RR^n$ have been recently widely used to model diffusion processes. More precisely, in \cite{Fi} the authors consider some $u(x, t)$ that models the  probabilist density function of a single population at the point $x$ at time $t$. Let $J$ be a symmetric function  with $\int_{\RR^n}J(x)\,dx=1$; $J(x-y)$ is the probability distribution of jumping from location $y$ to location $x$; $J\ast u(x,t) = \int_{\mathbb{R}^n}J(y-x)u(y,t)\,dy$ is the rate at which individuals are arriving to position $x$ from all other places, and $-u(x,t)=\int\limits_{\RR^{n}}J(x-y)u(x,t)dy$ is the rate at which they are leaving location $x$ to travel to all other sites. Then $u$ satisfies a nonlocal evolution equation of the form
\begin{align}\label{0.1}
u_t(x,t) & =J\ast u(x,t)-u(x,t).
\end{align}

In the work \cite{CER} the authors prove that solutions of properly rescaled nonlocal Dirichlet problems of the equation \eqref{0.1}  uniformly approximate the solution of the corresponding Dirichlet problem for the classical heat equation in $\RR^n$.

These type of problems have been used to model very different applied situations,
for example in biology \cite{CF} and \cite{ML}, image processing \cite{KOJ}, particle systems \cite{BV},
coagulation models \cite{FL}, etc.

In the context of the Euclidean space $\RR^n$ some of these results have been generalized for kernels that are not convolution. It is interesting how this kind of problems translate to other settings. For example, in \cite{Vi} the author considers a nonlocal diffusion problem on the Heisenberg group and analogous results to those obtained in the works \cite{CCR} and \cite{CER}. In our work we will consider the following problems (which are originally set in $\mathbb{R}^{n}$), in the more general context of the Carnot groups (let us recall that both the Euclidean space and the Heisenberg group are examples of Carnot groups):
\begin{itemize}
\item  In the work \cite{MR} the authors prove that smooth solutions to the Dirichlet problem for the parabolic equation
$$
v_t(x,t)=\sum_{i,j}^na_{i,j}(x)\frac{\partial^2 v(x,t)}{\partial x_i\partial x_j}+\sum_i^nb_i\frac{\partial v(x,t)}{\partial x_i}, \qquad x\in\Omega,
$$
with $v(x,t) = g(x, t)$, $x \in\partial\Omega$, can be uniformly approximated by solutions of nonlocal problems of the form 
$$
u_t^\epsilon(x,t)=\int_{\RR^n}K_{\epsilon}(x,y)(u^\epsilon(y,t)-u^\epsilon(x,t))\,dy, \qquad x\in \Omega
$$
with $u^\epsilon(x,t) = g(x, t)$ for $x\notin \Omega$ as $\epsilon \to 0$, for an appropriate rescaled kernel $K_\epsilon$.

\item On the other hand, in \cite{SLY} the authors consider the next Fokker-Planck equation
$$v_{t}(x,t)=\sum\limits_{i=1}^{n} (a(x)v(x,t))_{x_{i}x_{i}},\qquad x\in \Omega,$$
with $u(x,t) = g(x, t)$ for $x\notin \Omega$ and the coefficients $a\in C^{\infty}(\RR^{n})$.
They prove that the solutions of this problem can be uniformly approximated by the solutions of the non-local problem
$$u_{t}(x,y)=\int\limits_{R^{n}} a(y)J(x-y)u(y,t)dy - a(x)u(x,t),\qquad x\in \Omega, $$  
properly rescaled, where $\int\limits_{\RR^{n}}J(x)dx=1$ and $u(x,t) = g(x, t)$ for $x\notin \Omega$.

\end{itemize}

In this way, in \cite{MR} and \cite{SLY} the authors show that the usual local
evolution problems with spatial dependence can be approximated by nonlocal ones.

The study of Carnot groups and PDE's on them has been increasing in the last years, since the topology is similar to the Euclidean topology and the hypoelliptic equations are easily defined (see the fundamental work of H\"{o}rmander \cite{H}). Regularity results, study of fundamental solutions, computation of a priori estimates, study of asymptotic behaviour, and other aspects of PDEs in this context, and mainly for the subLaplacian and for the heat operator, can be found in, for example, the works \cite{BLU}, \cite{BBLU}, \cite{CG}, \cite{BF},  \cite{DG}, \cite{R}, and references therein. Let us remark that this list is by no means exhaustive, since the literature on these matters is really extensive.

A Carnot group is a simply connected and connected Lie group $G$, whose Lie algebra $\mathfrak{g}$ is stratified, this means that $\mathfrak{g}$ admits a vector space decomposition $\mathfrak{g}=V_1\oplus \cdots\oplus V_m$ with grading $[V_1,V_j]=V_{j+1}$, for $j=1,\cdots, m-1,$ and has a family of dilations $\{\delta_\epsilon\}_{\epsilon>0}$ such that $\delta_\epsilon X=\epsilon^j X$ if $X\in V_j$.

Let be $\{X_1,\dots,X_{n_1}\}$ a basis of $V_1$ and $\{X_{n_1+1},\dots,X_{n_1+n_2}\}$ a basis of $V_2$ and let $\Omega\subset G$ be a bounded $C^{2+\alpha}$, $0<\alpha<1$, domain (that is, open and connected). The smoothness condition on $\Omega$ means that the boundary is the graph of a $C^{2+\alpha}$ function, which precise definition will be given in section \ref{subsec:Str.T.Ineq}.

We consider the following second order local parabolic differential equation with Dirichlet boundary conditions
\begin{equation}\label{0.2}%
\begin{cases}
v_{t}(x,t)=\sum\limits_{i=1}^{n_1}\sum\limits_{j=1}^{n_1}a_{ij}(x)X_i X_j v(x,t)+\sum\limits_{i=1}^{n_1+n_2}b_{i}(x)X_iv(x,t),\quad & x\in \Omega, t>0,\\
v(x,t)=g(x,t), & x\in \partial\Omega, t>0,\\
v(x,0)=u_{0}(x), & x\in\Omega,
\end{cases}
\end{equation}
where the coefficients $a_{ij}(x)$, $b_i(x)$ are smooth in $\overline\Omega$ and $(a_{ij}(x))$ is a symmetric positive definite matrix, i.e., $\sum_{i,j} a_{i,j}(x)\xi_i\xi_j\geq \nu|\xi|^2$ for every real vector $\xi = (\xi_1,\dots,\xi_{n_1}) \neq 0$ and for some $\nu>0$. Also we have the following nonlocal rescaled Dirichlet problem
\begin{equation}\label{0.3}%
\begin{cases}
u^{\epsilon}_{t}(x,t)=\mathcal{K}_{\epsilon}(u^\epsilon)(x,t),\qquad & x\in \Omega,\quad t>0,\\
u^{\epsilon}(x,t)=g(x,t), & x\notin \Omega,\quad t>0,\\
u^{\epsilon}(x,0)=u_{0}(x), & x\in\Omega,
\end{cases}
\end{equation}
where $\mathcal{K}_{\epsilon}$ is a nonlocal operator defined by a rescaled kernel (see section \ref{subsec:EvolEqn.kernel}). We will prove the next Theorem:

\begin{thm}\label{thm:MainMR}
Let $u^{\epsilon}$ be the solution of problem \eqref{0.3} where $\mathcal{K_{\epsilon}}$ is defined by formula \eqref{Kepsilon}, $g\in C^{2+\alpha,1+\alpha}(\Omega^{c}\times [0,T])$ and $u_{0}\in C^{2+\alpha}(\Omega)$. Then there exists a positive constant $c$ such that 
\begin{align*}
 ||u^{\epsilon}-v||_{L^{\infty}(\Omega\times[0,T])}\le c\epsilon^{\alpha},
\end{align*}
where $v$ is the solution of problem \eqref{0.2}.
\end{thm}

We also study the  Fokker-Planck parabolic problem with Dirichlet condition
\begin{equation}\label{0.4}%
\begin{cases}
v_{t}(x,t)=\sum\limits_{i=1}^{n_1}X_{i}X_{i}(a(\cdot)v(\cdot,t))(x),\qquad & x\in \Omega,\quad t>0,\\
v(x,t)=g(x,t), & x\in \partial\Omega,\quad t>0,\\
v(x,0)=u_{0}(x), & x\in\Omega,
\end{cases}
\end{equation}
where the coefficient $a\in C^{\infty}(G)$; and the nonlocal reescaled Dirichlet problem given by 
\begin{equation}\label{0.5}%
\begin{cases}
u^{\epsilon}_{t}(x,t)=\mathcal{L}_{\epsilon}(u^\epsilon)(x,t),\qquad & x\in \Omega,\quad t>0,\\
u^{\epsilon}(x,t)=g(x,t), & x\notin \Omega,\quad t>0,\\
u^{\epsilon}(x,0)=u_{0}(x), & x\in\Omega,
\end{cases}
\end{equation}
where $\mathcal{L}_{\epsilon}$ is defined in section \ref{subsec:SYL}.  We will prove the next Theorem:
\begin{thm}\label{thm:MainSYL}
Let $u^{\epsilon}$ be the solution of problem \eqref{0.5} where $\mathcal{L_{\epsilon}}$ is defined by formula \eqref{Lepsilon}, $g\in C^{2+\alpha,1+\alpha}(\Omega^{c}\times [0,T])$ and $u_{0}\in C^{2+\alpha}(\Omega)$. Then there exists a positive constant $c$ such that 
\begin{align*}
 ||u^{\epsilon}-v||_{L^{\infty}(\Omega\times[0,T])}\le c\epsilon^{\alpha},
\end{align*}
where $v$ is the solution of problem \eqref{0.4}.
\end{thm}
 
Such results as those of Theorems \ref{thm:MainMR} and \ref{thm:MainSYL} allows us to approximate the solutions of flow equations in Carnot groups.

It is important to stress that here we will use that \eqref{0.2} and \eqref{0.4} have smooth solutions. In fact, under regularity assumptions on the boundary data $g$, the domain $\Omega$ and the initial condition $u_0$, we have that the solutions of \eqref{0.2} are $C^{2+\alpha,1+\alpha/2}(\Omega\times[0,T])$. For such a regularity result, we refer to the previously cited articles in page 3 and also the works \cite{WLL} and  \cite{XZ}.

\par The rest of the paper is organized as follows. In section \ref{sec:Prelim} we recall some definitions and results on Carnot groups and set the notation to be used later. In section \ref{sec:kernels} we define and study the operators $\mathcal{K}_{\epsilon}$ and $\mathcal{L}_{\epsilon}$. In section \ref{sect3.1} we study the existence, uniqueness and other properties of the solutions of the problems \eqref{0.3} and \eqref{0.5}. In section \ref{sec:Proofs} we prove the Main Theorems \ref{thm:MainMR} and \ref{thm:MainSYL}.

\section{Preliminaries}\label{sec:Prelim}

\subsection{Homogeneous Lie groups}\label{subsec:Hom.L.Grps.}

\par Let $\mathfrak{g}$ be a real Lie algebra of finite dimension $n$ and let $G$ be its corresponding connected and simply connected Lie group. Recall that if $G$ is nilpotent, the exponential map $\exp:\mathfrak{g}\to G$ is a diffeomorphism.

\par If $\mathfrak{g}$ is nilpotent and we choose a basis $\{X_{1},\dots,X_{n}\}$ for $\mathfrak{g}$, we can identify $\mathbb{R}^{n}$ with the group $G$ via the exponential map: let $\varphi:\mathbb{R}^{n}\to G$ be such that every $(x_{1},\dots,x_{n})\in\mathbb{R}^{n}$ is identified with $\varphi(t_{1},\dots,t_{n})=\exp(t_{1}X_{1}+\dots+t_{n}X_{n})$. Observe that $\varphi^{-1}$ defines a global chart on the Carnot group $G$. Since the Campbell-Hausdorff-Baker series has finitely many terms, the group law is a polynomial map and may be written as $xy=(p_{1}(x,y),\dots,p_{n}(x,y))$, where the $p_{j}$ are polynomials maps.

\par A \textit{family of dilations} on $\mathfrak{g}$ is a one parameter family of automorphisms $\{\delta_{r}:r>0\}$ of $\mathfrak{g}$  of the form $\delta_{r}=\Exp(A\log r)$, where $\Exp$ denotes the matrix exponential function, and  $A$ is a diagonalizable linear transformation on $\mathfrak{g}$ with positive eigenvalues. Thus, \begin{align*}
\delta_{r}(X)=&\Exp(AX\log r)=\sum\limits_{l=0}^{\infty}\frac{1}{l!}(\log r AX)^{l}.\end{align*} 
 If a Lie algebra admits a family of dilations then $\mathfrak{g}$ is nilpotent (see \cite{FS} for example), and $G$ is called an \textit{homogeneous Lie group}. Let us remark that not every nilpotent Lie algebra admits a family of dilations (see \cite{D}).

\par If $G$ is an homogeneous Lie group it is nilpotent, hence the dilations on $\mathfrak{g}$ lift via the exponential map to a one parameter group of automorphisms on $G$. We also call the maps $\exp \circ \, \delta_{r} \circ (\exp)^{-1}$ \textit{dilations} on $G$ and denote them again by $\delta_{r}$. The \textit{homogeneous dimension} of $G$ is defined to be the number $Q=trace(A)=\lambda_{1}+\dots +\lambda_{n}$ where $\lambda_{1},\dots,\lambda_{n}$ are the eigenvalues of $A$. Since $\Exp(\alpha A \log r)=\delta_{r^{\alpha}}$ for any $\alpha>0$, we may assume that the smallest eigenvalue of $A$ is $1$, and moreover, we may assume that $1=\lambda_{1}\le\dots\le \lambda_{n}=\overline{\lambda}$. Let us choose $\beta=\{X_{1},\dots,X_{n}\}$ a basis of eigenvectors  of $A$  with $AX_{j}=\lambda_{j} X_{j}$ where the automorphism $A$ is diagonal: $$[A]_{\beta}=\begin{pmatrix} \lambda_{1} & \dots & 0 \\ \vdots & \ddots & \vdots \\ 0 & \dots & \lambda_{n} \end{pmatrix},$$ hence also, since the dilations are automorphisms, $$[\delta_{r}]_{\beta}=\begin{pmatrix} r^{\lambda_{1}} & \dots & 0 \\ \vdots & \ddots & \vdots \\ 0 & \dots & r^{\lambda_{n}} \end{pmatrix},$$
then $r^{\lambda_{j}}$ is an eigenvalue for $\delta_{r}$ and $X_{j}$ is an associated eigenvector. We have that $\delta_{r}[X_{i},X_{j}]=[\delta_{r}X_{i},\delta_{r}X_{j}]=r^{\lambda_{i}+\lambda_{j}}[X_{i},X_{j}]$.

\par A Lie algebra $\mathfrak{g}$ is called \textit{graded} if it is endowed with a vector space decomposition $\mathfrak{g}=\bigoplus\limits_{j=1}^{\infty}V_{j}$ (where all but finitely many of the $V_{k}$'s are $\{0\}$), such that $[V_{i},V_{j}]\subset V_{i+j}$. The Lie group $G$ is also called \textit{graded}. 
\par  A graded Lie algebra $\mathfrak{g}$ is said to be \textit{stratified} if it admits a vector space decomposition as follows: there exists $m\le n$ such that $\mathfrak{g}=V_{1}\oplus\dots\oplus V_{m}$, where $V_{k+1}=[V_{1},V_{k}]\neq\{0\}$, for all $1\le k< m$ and $V_{k}=\{0\}$ if $k>m$. This also means that $V_{1}$ generates $\mathfrak{g}$ as a Lie algebra. A stratified Lie algebra is nilpotent of step $m$ and there is a natural family of dilations on $\mathfrak{g}$ given by $\delta_{r}\left(\sum\limits_{k=1}^{m}Y_{k}\right)=\sum\limits_{k=1}^{m}r^{k}Y_{k}$, where each $Y_{k}\in V_{k}$. The Lie group $G$ is also called \textit{stratified} or \textit{Carnot group}.

\par In the case of a stratified Lie algebra $\mathfrak{g}$, the following notation will be used:
\begin{itemize}
\item the set of eigenvalues of $A$ is $\mathcal{A}=\{1,2,\dots,m\}$,
\item the set of eigenvalues for each $\delta_{r}$, $r>0$, is $\{r^{1},\dots,r^{m}\}$,
\item the basis $\beta=\{X_{1},\dots,X_{n}\}$ of $\mathfrak{g}$ is adapted to the gradation in the following sense: if $\dim(V_{k})=n_{k}$ for $1\le k\le m$, then $n=n_{1}+\dots+n_{m}$, and 
	\begin{itemize}
		\item $\beta_{1}=\{X_{1},\dots,X_{n_{1}}\}$ is a basis of $V_{1}$ of eigenvectors associated to the eigenvalue $\lambda=1$, 
		\item $\beta_{2}=\{X_{n_{1}+1},\dots,X_{n_{1}+n_{2}}\}$ is a basis of $V_{2}$ of eigenvectors associated to the eigenvalue $\lambda=2$, 
		\item $\dots$
		\item $\beta_{m}=\{X_{n_{1}+\dots+n_{m-1}+1},\dots,X_{n}\}$ is a basis of $V_{m}$ of eigenvectors associated to the eigenvalue $\lambda=m$.
	\end{itemize}
\item the homogeneous dimension is $Q=\sum\limits_{k=1}^{m}kn_{k}$.
\item the identification $\phi=(\phi_{1},\dots,\phi_{n}):\mathfrak{g}\to\mathbb{R}^{n}$ such that if $X=t_{1}X_{1}+\dots + t_{n}X_{n}\in\mathfrak{g}$, then $\phi(X)=(t_{1},\dots,t_{n})$, and $\phi_{j}(X)=t_{j}$. Observe that if $x\in G$, $\phi_{j}(\exp^{-1}(x))=\pi_{j}(\varphi^{-1}(x))$, where $\pi_{j}:\mathbb{R}^{n}\to\mathbb{R}$ denotes the projection and $\varphi$ is the inverse of the global chart on the group. 
\end{itemize}

\par Let us consider until the end of the section an homogeneous Lie algebra $\mathfrak{g}$ with homogeneous Lie group $G$. 

\par We define an Euclidean norm $||.||$ on  $G$ as follows: if $\{X_{1},\dots,X_{n}\}$ is a basis of $\mathfrak{g}$, just define $||.||$ by establishing that $X_{i}$ and $X_{j}$ are orthogonal for all $1\le i, j \le n$, $i\neq j$, then lift it to $G$ via the exponential map, that is, for $x\in G$ define $||x||=||\exp^{-1}x||$. For practical reasons we will use an \textit{homogeneous norm} $|.|$ which we construct as follows:  if $X=\sum\limits_{1}^{n}t_{j}X_{j} \in \mathfrak{g}$ then $||\delta_{r}X||=\left( \sum\limits_{1}^{n} t_{j}^{2} r^{2\lambda_{j}} \right)^{\frac{1}{2}}$. If $X\neq 0$ then $||\delta_{r}X||$ is a strictly increasing function of $r$ which tends to $0$ or $\infty$ along with $r$. Hence there is a unique $r(X)>0$ such that $||\delta_{r(X)}X||=1$, and we set $|0|=0$ and $|x|=\frac{1}{r}||\exp^{-1}x||$ for $x\neq 0$. 

\par The Lebesgue measure on $\mathfrak{g}$ induces a biinvariant Haar measure on $G$, and we fix the normalization of Haar measure on $G$ by requiring that the measure of the unitary ball to be $1$. We shall denote with $|E|$ the measure of a measurable set $E$ and with $\int f=\int f dx$ the integral of a function $f$ with respect to this measure. Hence, $|\delta_{r}(E)|=r^{Q}|E|$ and $d(rx)=r^{Q}dx$. In particular, $|B(r,x)|=r^{Q}$ for all $r>0$ and $x\in G$.

\par A function $f$ on $G\backslash\{0\}$ will be called \textit{homogeneous of degree $\sigma$} if $f\circ\delta_{r}=r^{\sigma}f$ for $r>0$. For any $f$ and $g$ we have that $\int f(x)(g\circ\delta_{r})(x)dx=r^{-Q}\int(f\circ\delta_{\frac{1}{r}})(x)g(x)dx$, if the integrals exist. Hence we extend the map $f\to f\circ{\delta_{r}}$ to distributions as follows: $<f\circ\delta_{r},\varphi>=r^{-Q}<f,\varphi\circ\delta_{\frac{1}{r}}>$, for a distribution $f$ and a test function $\varphi$. We say that a distribution is \textit{homogeneous of degree $\sigma$} if $f\circ_{\delta_{r}}=r^{\sigma}f$. A differential operator $D$ on $G$ is \textit{homogeneous of degree $\sigma$} if $D(f\circ\delta_{r})=r^{\sigma}(Df)\circ\delta_{r}$ for any $f$. Observe that if $D$ is homogeneous of degree $\sigma$ and $f$ is homogeneous of degree $\mu$ then $Df$ is homogeneous of degree $\mu-\sigma$.

\par The approximations to the identity in this context take the following form: if $\psi$ is a function on $G$ and $t>0$, we define $\psi_{t}=t^{-Q}\psi\circ\delta_{\frac{1}{t}}$. Observe that if $\psi\in L^{1}$ then $\int\psi_{t}(x)dx$ is independent of $t$.

\par We will also use the following multiindex notation: if $I=(i_{1},\dots,i_{n})\in\mathbb{N}_{0}^{n}$, we set $X^{I}=X_{1}^{i_{1}}X_{2}^{i_{2}}\dots X_{n}^{i_{n}}$. The operators $X^{I}$ form a basis for the algebra of left invariant differential operators on $G$, by the Poincar\'{e}-Birkhoff-Witt Theorem. The order of the differential operators $X^{I}$ is $|I|=i_{1}+i_{2}+\dots+i_{n}$ and its \textit{homogeneous degree} is $d(I)=\lambda_{1}i_{1}+\lambda_{2}i_{2}+\dots +\lambda_{n}i_{n}$. Finally, let $\triangle$ be the additive subsemigroup of $\mathbb{R}$ generated by $0,\lambda_{1},\dots,\lambda_{n}$. Observe that $\triangle=\{d(I):I\in\mathbb{N}^{n}\}\supset\mathbb{N}$ (since $\lambda_{1}=1$), and if $G$ is a Carnot group $\triangle=\mathbb{N}$.

\par Finally, if $G$ is a Carnot group, it is clear that $X\in\mathfrak{g}$ is homogeneous of degree $k$ if and only if $X\in V_{k}$. We have defined the basis $\beta$ of eigenvectors such that $X_{1},\dots,X_{n_{1}}$ is a basis for $V_{1}$. Let us now define $\mathcal{J}=\sum\limits_{j=1}^{n_{1}}X_{j}^{2}$, thus $-\mathcal{J}$ is a left invariant second order differential operator which is homogeneous of degree 2 called the subLaplacian of $G$ (relative to the stratification and the basis). Its role on $G$ is analogous to (minus) the ordinary Laplacian in $\mathbb{R}^{n}$.

\subsection{Taylor polynomials in homogeneous Lie groups}

\par Now we are going to recall some concepts and notations on the definition of Taylor polinomials for homogeneous Lie groups from \cite{FS}. We say that a function $P$ on $G$ is a \textit{polynomial} if $P\circ\exp$ is a polynomial on $\mathfrak{g}$. Let $\{\xi_{1},\dots,\xi_{n}\}$ be the basis for the linear forms on $\mathfrak{g}$ dual to the basis $\{X_{1},\dots,X_{n}\}$ on $\mathfrak{g}$. Let us consider $\eta_{j}=\xi_{j}\circ\exp^{-1}$, then $\eta_{1},\dots,\eta_{n}$ are polynomials on $G$ which form a global coordinate system on $G$, and generate the algebra of polynomials on $G$. Thus, every polynomial on $G$ can be written uniquely as $P=\sum_{I}a_{I}\eta^{I}$, for $\eta^{I}=\eta_{1}^{i_{1}}\dots\eta_{n}^{i_{n}}$, $a_{I}\in\mathbb{C}$ where all but finitely many of them vanish. Since $\eta^{I}$ is homogeneous of degree $d(I)$, the set of possible degrees of homogeneity for polynomials is the set $\triangle$. We call the degree of a polynomial $\max\{|I|:a_{I}\neq 0\}$ the \textit{isotropic degree}, and its \textit{homogeneous degree} is $\max\{d(I):a_{I}\neq 0\}$. For each $N\in\mathbb{N}$ we define the space $\mathcal{P}_{N}^{iso}$ of polynomials of isotropic degree $\le N$, and for each $j\in\triangle$ we define the space $\mathcal{P}_{j}$ of polynomials of homogeneous degree $\le j$. It follows that $\mathcal{P}_{N}\subset \mathcal{P}_{N}^{iso}\subset \mathcal{P}_{\overline{\lambda}N}$. The space $\mathcal{P}_{j}$ is invariant under left and right translations (see Prosition 1.25 of \cite{FS}), but the space $\mathcal{P}_{N}^{iso}$ is not (unless $N=0$ or $G$ is abelian). For a function $f$ whose derivatives $X^{I}f$ are continuous functions on a neighbourhood of a point $x\in G$, and for $j\in\triangle$ such that $d(I)\le j$ we define the \textit{left Taylor polynomial of $f$ at $x$ of homogeneous degree $j$} to be the unique polynomial $P\in\mathcal{P}_{j}$ such that $X^{I}P(e)=X^{I}f(x)$. Here, we have that $X^{I}f(x) = \frac{\partial^{d(I)}}{\partial^{i_{1}}t_{1}\dots\partial^{i_{n}}t_{n}} \left.f\left(x\exp\sum\limits_{j=1}^{n}t_{j}X_{j}\right)\right|_{t_{1}=\dots=t_{n}=0}$. From now on we are going to consider every Taylor polynomial as a left Taylor polynomial, hence we will drop the word \textit{left}.

\par We will be using the Taylor polynomial of a function $f$ at a point $x\in G$ of homogeneous degree 2, hence we will explicitly show its form with the notation we have presented. Let us call it $P=P_{f,x,2}=\sum\limits_{I\in\mathbb{N}^{n}:d(I)\le 2}a_{I}\eta^{I}$. If the multiindex $I\in\mathbb{N}^{n}$ has homogeneous degree $d(I)=0$, $I=\overline{0}$. If $d(I)=1$ then $I=\overline{e_{j}}$ for $1\le j \le n_{1}$ (where $\overline{e_{j}}$ denotes the canonical unitary vectors of $\mathbb{R}^{n}$, whose $i-$component is defined as $\delta_{ij}$), and if $d(I)=2$ then either $I=\overline{e_{i}}+\overline{e_{j}}$ for $1\le i,j \le n_{1}$ or $I=\overline{e_{j}}$ for $n_{1}+1\le j \le n_{1}+n_{2}$. Hence, for $y\in G$, $$P(y)=a_{\overline{0}} + \sum\limits_{j=1}^{n_{1}+n_{2}} a_{\overline{e_{j}}} \eta^{\overline{e_{j}}}(y) + \sum\limits_{i,j=1}^{n_{1}} a_{\overline{e_{i}}+\overline{e_{j}}} \eta^{\overline{e_{i}}+\overline{e_{j}}}(y).$$ In order to explicitly state the coefficients we perform some straightforward computations, namely:
\begin{itemize}
\item[-] If $I=\overline{0}$ then since $X^{\overline{0}}P(e)=X^{\overline{0}} f(x)$ we have that $a_{\overline{0}}=f(x)$.
\item[-] Let us consider $I=\overline{e_{j}}$ for $1\le j \le n_{1}+n_{2}$. Then for $y=\exp \sum\limits_{l=1}^{n}t_{l}X_{l}$ we have that $\eta^{\overline{e_{j}}}(y)=(\xi_{j}\circ\exp^{-1})(y)=t_{j}=(\pi_{j}\circ\varphi^{-1})(y)$. Since $X^{\overline{e_{j}}}P(e)=X^{\overline{e_{j}}} f(x)$ it follows that $a_{\overline{e_{j}}}=X_{j}f(x)$.
\item[-] Similarly, if we consider $I=\overline{e_{i}}+\overline{e_{j}}$ for $1\le i,j \le n_{1}$ we have that $\eta^{\overline{e_{i}}+\overline{e_{j}}}(y)=t_{i}t_{j}=(\pi_{i}\circ\varphi^{-1})(y)(\pi_{j}\circ\varphi^{-1})(y)$. And from the equality $X^{\overline{e_{i}}+\overline{e_{j}}}P(e)=X^{\overline{e_{i}}+\overline{e_{j}}} f(x)$ it follows that $a_{\overline{e_{i}}+\overline{e_{j}}}+a_{\overline{e_{j}}+\overline{e_{i}}}=X_{i}X_{j}f(x)$.
\end{itemize}
We are now able to present the Taylor polynomial $P$ in a more familiar form: $$P(y)=f(x) + \sum\limits_{j=1}^{n_{1}+n_{2}} (\pi_{j}\circ\varphi^{-1})(y) X_{j}f(x)  + \frac{1}{2} \sum\limits_{i,j=1}^{n_{1}} (\pi_{i}\circ\varphi^{-1})(y)(\pi_{j}\circ\varphi^{-1})(y) X_{i}X_{j}f(x).$$ And if we are considering coordinates,
\begin{align}\label{Taylor.2}
P\left(\exp \left(\sum\limits_{l=1}^{n}t_{l}X_{l}\right)\right) = & f(x) + \sum\limits_{j=1}^{n_{1}+n_{2}} t_{j} X_{j}f(x)  + \frac{1}{2} \sum\limits_{i,j=1}^{n_{1}} t_{i}t_{j} X_{i}X_{j}f(x).
\end{align}

\subsubsection{Stratified Taylor inequality}\label{subsec:Str.T.Ineq}

\par Throughout this section we will consider a fixed stratified group $G$ with the notation described previously. We will regard the elements of the basis of $\mathfrak{g}$ adapted to the gradation as left invariant differential operators on $G$.

\par Since $V_{1}$ generates $\mathfrak{g}$ as a Lie algebra, we have that $\exp(V_{1})$ generates $G$. More precisely:
\begin{lem}[Lemma 1.40 of \cite{FS}]\label{Lem:strat.coords}
If $G$ is stratified there exist $C>0$ and $N\in\mathbb{N}$ such that any $x\in G$ can be expressed as $x=x_{1}\dots x_{N}$ with $x_{j}\in\exp(V_{1})$ and $|x_{j}|\le C|x|$, for all $j$.
 \end{lem}

\par For $k\in\mathbb{N}$ we define $C^{k}(G)$ to be the space of continuous functions $f$ on $G$ whose derivatives $X^{I}f$ are continuous functions on $G$ for $d(I)\le k$. Also, for $0<\alpha<1$ we define the space $C^{k+\alpha}(G)$ as the function $f$ in $C^{k}(G)$ where
$$
\sup_{x,y\in G, d(I)=k}|X^If(xy)-X^If(x)|<C |y|^\alpha,
$$
with $C$ independent of $x$ and $y$.

\par Let us also define the space $C^{k}(\Omega)$ of those functions $f$ defined on $\Omega$ such that $Df$ is continuous for every differential operator $D$ of homogeneous degree less or equal to $k$.

\par Another important consequence of the fact that $V_{1}$ generates $\mathfrak{g}$ is that the set of left invariant differential operators which are homogeneous of degree $k$ (which is the linear span of $\{X^{I}:d(I)=k\}$) is precisely the linear span of the operators $X_{i_{1}}\dots X_{i_{k}}$ with $1\le i_{j}\le n_{1}$ for $j=1,\dots,k$. We thus have the following results:
\begin{thm}[Theorem 1.41, Stratified Mean Value Theorem, \cite{FS}]\label{Thm:SMVT} \noindent \par Suppose $G$ is stratified. There exist $C>0$ and $b>0$ such that for all $f\in C^{1}$ and all $x,y\in G$, $$|f(xy)-f(x)|\le C|y|\sup\limits_{|z|\le b|y|, 1\le k\le n_{1}}|X_{k}f(xz)|.$$
\end{thm}
\begin{thm}[Theorem 1.42, Stratified Taylor Inequality, of \cite{FS}]\label{Thm:STI} \noindent \par Suppose $G$ is stratified. For each positive integer $k$ there is a constant $C_{k}$ such that for all $f\in C^{k}$ and all $x,y\in G$, $$|f(xy)-P_{x}(y)|\le C_{k}|y|^{k}\eta(x,b^{k}|y|),$$ where $P_{x}$ is the Taylor polynomial of $f$ at $x$ of homogeneous degree $k$, $b$ is as in the Stratified Mean Value Theorem, and for $r>0$, $$\eta(x,r)=\sup\limits_{|z|\le r, d(I)=k} |X^{I}f(xz)-X^{I}f(x)|.$$
\end{thm}

\par For a function $f\in C^{k}(\Omega)$ and $x\in\Omega$ let $P=P_{f,x,k}$ denote the Taylor polynomial of $f$ at $x$ of homogeneous degree $k$. By Theorem \ref{Thm:STI} we have that for $\epsilon>0$,
\begin{align}\label{vani}
\frac{1}{\epsilon^{k}}|f(x\delta_{\epsilon} y)-P(y)|\le & \frac{c_{k}|\delta_{\epsilon}y|^{k}}{\epsilon^{k}} \eta(x,b^{k}|\delta_{\epsilon}y|) = c_{k} |y|^{k} \eta(x,b^{k}|\delta_{\epsilon}y|),
\end{align}
which goes to $0$ as $\epsilon$ does.

\par Hence,  if $f$ in $C^{2+\alpha}(G)$, with $0<\alpha<1$ we have the following Taylor expansion of $f$ at $x$ of homogeneous degree $k=2$: for 
\begin{align}\label{simple.Taylor.expansion} \nonumber
f\left(x\exp\left(\sum\limits_{l=1}^{n}t_{l}X_{l}\right)\right)= & f(x)+\sum\limits_{j=1}^{n_{1}+n_{2}}t_{j}X_{j}f(x) \\+ &  \frac{1}{2} \sum\limits_{i,j=1}^{n_{1}}t_{i}t_{j}X_{i}X_{j}f(x)+o(|t_{1}X_{1}+\dots+t_{n}X_{n}|^{2}),
\end{align}
in the sense that
\begin{align*}
\lim\limits_{\epsilon\to 0} \frac{o(|\delta_{\epsilon}(t_{1}X_{1}+\dots+t_{n}X_{n})|^{2})}{\epsilon^{2}}=0.&
\end{align*}
Indeed, by \eqref{Taylor.2}, \eqref{vani} and the fact that $f\in C^{2+\alpha}(G)$, we get
\begin{align}\label{vani2}
&\left|\frac{o(|\delta_{\epsilon}(t_{1}X_{1}+\dots+t_{n}X_{n})|^{2})}{\epsilon^{2}}\right|  \\\nonumber
& \leq  c |\exp(t_{1}X_{1}+\dots+t_{n}X_{n})|^{2} \eta(x,b^{2}|\exp(\delta_{\epsilon}(t_{1}X_{1}+\dots+t_{n}X_{n}))|)  \\\nonumber
& = c |\exp(t_{1}X_{1}+\dots+t_{n}X_{n})|^{2} \sup\limits_{|z|\le b^{2}|\exp(\delta_{\epsilon}(t_{1}X_{1}+\dots+t_{n}X_{n}))|, d(I)=2} |X^{I}f(xz)-X^{I}f(x)| \\\nonumber
& \leq c |\exp(t_{1}X_{1}+\dots+t_{n}X_{n})|^{2}  b^{2}|\exp(\delta_{\epsilon}(t_{1}X_{1}+\dots+t_{n}X_{n}))|^\alpha  \\\nonumber
& = c b^{2}|\exp(t_{1}X_{1}+\dots+t_{n}X_{n})|^{2+\alpha} \epsilon^\alpha.
\end{align}

\par Before we move forward, let us remark that throughout the work we will denote with $c$ a positive constant that may vary from line to line.

\section{Some nonlocal diffusion problems}\label{sec:kernels}

\par Throughout this section we let $G$ be a Carnot group with Lie algebra $\mathfrak{g}$ and let $\Omega$ be an open, bounded and connected subset of $G$. The aim of this section is to properly define the operators $\mathcal{K}_{\epsilon}$ and $\mathcal{L}_{\epsilon}$ from the introduction. In order to understand the techniques involved, we will first work with an evolution operator of a much simpler form (namely the operator given in \eqref{0.1}), in the context of the Carnot group $G$. We will see that the solutions to the nonlocal Dirichlet rescaled problems uniformly approximate the solution of the classical heat equation with Dirichlet conditions.

\par Let us consider a positive and radial function $J\in L^{1}(G)$ with compact support $F$, normalized such that $\int_G J dx =1$,
whence for $i=1,\dots,n$
\begin{align}\label{J.x}
&\int\limits_{\mathbb{R}^{n}} J(\exp(t_{1}X_{1}+\dots+t_{n}X_{n})) t_{i} dt_{1}\dots dt_{n} =0;
\end{align}
and also
\begin{align}\label{J.x^2}
&\int\limits_{\mathbb{R}^{n}} J(\exp(t_{1}X_{1}+\dots+t_{n}X_{n})) t_{i}^{2} dt_{1}\dots dt_{n} = C(J),
\end{align}
for a constant $C(J)>0$, $i=1,\dots,n$. From both properties it follows that for $i,j=1,\dots,n$, 
\begin{align}\label{J.deltaij}
&\int\limits_{\mathbb{R}^{n}} J(\exp(t_{1}X_{1}+\dots+t_{n}X_{n})) t_{i}t_{j} dt_{1}\dots dt_{n} = C(J)\delta_{ij}.
\end{align}

\subsection{An evolution equation}\label{subsec:EvolEqn}

\par The evolution equation \eqref{0.1} is given in our context by the \textit{evolution operator} 
\begin{align}\label{evol.op}
\mathcal{E}u=J\ast u - u,
\end{align}
namely for a suitable domain $\Omega\times[0,T]$,
\begin{align}\label{evol.eqn}
u_{t}(x,t) =  & \mathcal{E}u(x,t).
\end{align}

\par For each $\epsilon >0$ we define the rescaled operator 
\begin{align}\label{reesc.evol.op}
\mathcal{E}_{\epsilon}u(x)=\frac{1}{\epsilon^{2}}\left[(u\ast J_{\epsilon})(x)-u(x)\right],
\end{align}
we have that
\begin{align} \nonumber
\mathcal{E}_{\epsilon}u(x)=&\frac{1}{\epsilon^{2}}\left[(u\ast J_{\epsilon})(x)-u(x)\right] = \frac{1}{\epsilon^{2}}\left[\int\limits_{G}u(xy^{-1})J_{\epsilon}(y)dy-u(x)\right] \\ \nonumber
=& \frac{1}{\epsilon^{2}}\left[\int\limits_{G}u(xy^{-1})\frac{1}{\epsilon^{Q}}J\left(\delta_{\frac{1}{\epsilon}}y\right)dy-u(x)\right] \\ \nonumber
=& \frac{1}{\epsilon^{2}}\left[\int\limits_{G}u(x(\delta_{\epsilon}(y))^{-1})\frac{\epsilon^{Q}}{\epsilon^{Q}}J(y)dy-\int\limits_{G}J(y)u(x)dy\right]\\ 
=& \frac{1}{\epsilon^{2}}\int\limits_{G}\left[u(x(\delta_{\epsilon}(y^{-1})))-u(x)\right]J(y)dy. \label{reesc.evol.eqn}
\end{align}

\subsection{An evolution equation given by a kernel}\label{subsec:EvolEqn.kernel}

\par In \cite{MR} Molino and Rossi studied the integral operator
\begin{align*}
\mathcal{K_\epsilon}u(x)=\int\limits_{G}K_\epsilon(x,y)(u(y)-u(x))dy,
\end{align*}
for $G=\mathbb{R}^{n}$, where the kernel $K_\epsilon(x,y)$ is a positive function with compact support in $\Omega\times \Omega$ for $\Omega\subset G$ a bounded domain such that $0<\sup\limits_{y\in\Omega}K_\epsilon(x,y)=c_\epsilon(x)\in L^{\infty}(\Omega)$.

\par Following the ideas of Molino and Rossi, let us consider:
\begin{itemize}
\item A function $J$ as defined in the begining of the section.
\item A $n_{1}\times n_{1}$ symmetric and positive definite matrix $\tilde A(x)=(a_{ij}(x))$, where the coefficients are smooth in $\overline\Omega$ with $\tilde A(x)=\tilde L(x)\tilde L^{t}(x)$ its Cholesky factorization, with $\tilde L(x)=(l_{ij}(x))$ and $\tilde L^{-1}(x)=(l^*_{ij}(x))$. Also, let $A(x)$ be the $n\times n$ matrix defined by blocks as follows: \begin{align*} A(x)= & \left( \begin{array}{c|c}\tilde A(x) & 0 \\ \hline 0 & \begin{array}{ccc} 1 & & 0 \\ & \ddots & \\ 0 & & 1 \end{array}  \end{array} \right). \end{align*} That is, $A(x)$ is the matrix $\tilde A(x)$ extended by the identity to size $n\times n$, and let $L(x)$ and $L^{t}(x)$ be similarly defined.
\item A $n\times n$ diagonal matrix $W(x)=\text{diag}(\tilde{b}_{1}(x),\dots,\tilde{b}_{n}(x))$ where $\tilde{b}_{i}(x)=b_{i}(x)$ if $1\le i\le n_{1}$, $\tilde{b}_{i}(x)=\frac{b_{i}(x)}{\epsilon^{2}}$ if $n_{1}<i\le n_{1}+n_{2}$ and $\tilde{b}_{i}(x)=1$ if $n_{1}+n_{2}<i\le n$.

\item A function $a:G\to \mathbb{R}$ defined by $a(x)=\sum\limits_{i=1}^{n} \phi_{i}(\exp^{-1}(x)) +M$, where $M>0$ is large enough to ensure $a(x)\ge \beta>0$ for $x$ belonging to an appropiate set $F'$ defined as 
\begin{align}\label{F'} F'&=\{x\in G: x=y\exp\delta_{\epsilon}L(y)\exp^{-1}(z^{-1}), \forall y\in\Omega, \forall z\in F\}, \end{align}
where $F$ is the support of $J$.
\end{itemize}

Thus, we will work with the scaled kernels defined for each $\epsilon>0$ by
\begin{align}\nonumber
K_{\epsilon}(x,y) = & \frac{c(x)}{\epsilon^{Q+2}}a((\exp(E(x)\exp^{-1}(y^{-1}x)))^{-1}) \\ \label{kernel}
& \qquad \times J(\exp(L^{-1}(x)\exp^{-1}(\delta_{\epsilon^{-1}}y^{-1}x))),
\end{align}
where for $x\in G$, $c(x)=\frac{2}{C(J)M(\det(A(x)))^{\frac{1}{2}}}$ and $E(x)=\frac{M}{2}W(x)A(x)^{-1}$. Let us remark that we understand the action of a $n\times n$ matrix $\textbf{M}$ on $\mathfrak{g}$ via the identification $\phi$ with $\mathbb{R}^{n}$ (with respect to the basis $\beta$) as follows: if $\textbf{M}=(m_{ij})$ and $X=\sum x_{i}X_{i}\in\mathfrak{g}$, $$\textbf{M}X=\sum\limits_{i=1}^{n}\left(\sum_{k=1}^{n} m_{ik}x_{k}\right)X_{i}.$$
Also, since the matrix $A(x)$, $L(x)$ and $W(x)$ are defined by blocks (with corresponding blocks of the same size), and the matrix which defines $\delta_{\epsilon}$ is also defined by blocks (again, of the same corresponding sizes) as a constant times the identity on each block, we have that $\delta_{\epsilon}$ commutes with all of them.

\par Hence, for these rescaled kernels we will study the integral operators
\begin{align}\nonumber
\mathcal{K}_{\epsilon}u(x)= & \frac{c(x)}{\epsilon^{Q+2}}\int\limits_{G}a((\exp(E(x)\exp^{-1}(y^{-1}x)))^{-1}) \\ \label{Kepsilon}
& \quad \times J(\exp(L^{-1}(x)\exp^{-1}(\delta_{\epsilon^{-1}}y^{-1}x)))(u(y)-u(x))dy.
\end{align}
More precisely, we will prove that $\mathcal{K}_{\epsilon}u$ approximates $\mathcal{K}v$ where $\mathcal{K}$ is the second order operator defined by
\begin{align}\label{operator}
\mathcal{K}(v)(x) = & \sum\limits_{i,j=1}^{n_{1}}a_{ij}(x)X_{i}X_{j}v(x)+\sum\limits_{i=1}^{n_{1}+n_{2}} b_{i}(x)X_{i}v(x).
\end{align}

\subsection{A reaction-diffusion equation}\label{subsec:SYL}

\par In \cite{SLY} the authors work in the same spirit as Molino and Rossi to approximate the solutions of the Fokker-Planck equation by solutions of operators defined by reescaled kernels which in our present context assume the form, respectively:
\begin{align}\label{operator1}
\mathcal{L}(v)(x)=\sum\limits_{i=1}^{n}X_{i}X_{i}(a(x)v(x)),
\end{align}
\begin{align}\label{Lepsilon}
\mathcal{L}_{\epsilon}(u)(x)=&\frac{2C(J)}{\epsilon^{Q+2}} \int\limits_{G} J(\delta_{\epsilon^{-1}}y^{-1}x)[a(y)u(y)-a(x)u(x)]dy,
\end{align}
with the coefficient $a\in C^{\infty}(G)$.

\section{Existence and uniqueness of solutions}\label{sect3.1}

We shall now derive the existence and uniqueness of
solutions of 
\begin{align}\label{pro2}
\left\{ \begin{array}{rclcc}
u^{\epsilon}_{t}(x,t) & = & \int_G K_\epsilon(x,y)\left(u(y,t)-u(x,t)\right)dy & \mbox{ for } & (x,t)\in\Omega\times[0,T],\\
u^{\epsilon}(x,t) & = & g(x,t) & \mbox{ for } & x\notin\Omega, t\in[0,T],\\
u^{\epsilon}(x,0) & = & u_{0}(x) & \mbox{ for } & x\in\Omega,
\end{array}
\right.
\end{align}
which is a consequence of Banach's fixed point theorem. The main arguments are basically the same of \cite{CCR} or \cite{CER}, but we write them here to make the paper self-contained. Let us also remark that the analogous results for operator $\mathcal{L}_{\epsilon}$ holds and the proofs are completely similar.

\par Recall the definition of the set $F'$ \eqref{F'}.

\begin{thm}\label{31}
Let $u_0\in L^1(\Omega)$ and let $J$ and $K_\epsilon$ defined as in Section \ref{subsec:EvolEqn.kernel}, with $K_\epsilon(x,y)\leq C_\epsilon(x)\in L^\infty(\Omega)$ for $(x,y)\in\Omega\times F'$. Then there exists a unique solution $u$ of problem \eqref{pro2} such that $u\in C([0,\infty),L^1(\Omega))$.
\end{thm}


\begin{proof}
We will use the Banach's Fixed Point Theorem. For $t_0 > 0$ let us consider the Banach space
\begin{align*}
X_{t_0}:=\left\{w\in C([0,t_0]; L^1 (\Omega))\right\},%
\end{align*}
with the norm
\begin{align*}
|||w|||:=\max_{0\leq t\leq t_0}\|w(\cdot,t)\|_{L^1(\Omega)}.
\end{align*}
Our aim is to obtain the solution of \eqref{pro2} as a fixed point of the operator $\mathfrak{T} : X_{ t_ 0} \rightarrow X_{ t_ 0}$ defined by
\begin{equation*}
\mathfrak{T} (w)(x,t):=
\begin{cases}
w_0(x)+\displaystyle\int_{0}^{t}\int_GK_\epsilon(x,y)\left(w(y,r)-w(x,r)\right)\,dydr \qquad & \text{if}\,\, x \in \Omega,\\
g(x,t) & \text{if}\,\, x \notin \Omega,
\end{cases}
\end{equation*} 
where $w_0(x)=w(x,0)$.

Let $w,\, v \in X_{t_ 0}$. Then there exists a constant $C$ depending on $K_\epsilon$ and $\Omega$ such that
\begin{align}\label{3.1}
|||\mathfrak{T}(w)-\mathfrak{T}(v)|||\leq Ct_0|||w-v|||+\|w_0-v_0\|_{L^1(\Omega)}.
\end{align} Indeed, since if $x\notin\Omega$ then $(w-v)(x,t)=0$, it follows that
\begin{align*}
&\int_{\Omega}\left|\mathfrak{T}(w)-\mathfrak{T}(v)\right|(x,t)dx\leq\int_{\Omega}|w_0-v_0|(x)dx\\
&\qquad \quad +\int_{\Omega}\left|\int_{0}^{t}\int_G K_\epsilon(x,y)\left((w-v)(y,r)-(w-v)(x,r)\right)\,dydr\right|dx\\
&\qquad \leq\|w_0-v_0\|_{L^1(\Omega)}+t\|C_\epsilon(x)\|_{L^\infty(\Omega)} 2 |\Omega| \; |||(w-v)|||.
\end{align*}
Taking the maximum in $t$  \eqref{3.1} follows.

Now, taking $v_0\equiv v\equiv 0$ in \eqref{3.1} we get that $\mathfrak{T}(w) \in C([0,t_0]; L^1 (\Omega))$
and this says that $\mathfrak{T}$ maps $X_{t_0}$ into $X_{t_0}$.  

Finally, we will consider $X_{t_0,u_0}=\{u\in X_{t_0}:\,u(x,0)=u_0(x)\}$. $\mathfrak{T}$ maps $X_{t_0,u_0}$ into $X_{t_0,u_0}$ and taking $t_ 0$ such that $2\|C_\epsilon(x)\|_{L^\infty(\Omega)}|\Omega|t_ 0 < 1$,
  we can apply the Banach's fixed point theorem in the interval $[0,t_0]$ because  $ \mathfrak{T}$ is a strict contraction in $X_{ t_ 0,u_0}$. From this we get the existence and uniqueness of the solution in $[0,t_0]$. To extend the solution to $[0,\infty)$ we may take as initial data $u(x, t_ 0 ) \in L^1 (\Omega)$ and obtain a solution up to $[0, 2t_ 0 ]$. Iterating this procedure we get a solution defined in $[0,\infty)$. 
\end{proof}

In order to prove a comparison principle of the problem given in (\ref{pro2}) we need to introduce the definition of sub and super solutions.
\begin{defin}
A function $u \in C([0, T ]; L^1 (\Omega))$ is a supersolution of \eqref{pro2} if
\begin{equation}\label{pro2sub}
\begin{cases}
u_t(x,t)\geq \int_G K_\epsilon(x,y)\left(u(y,t)-u(x,t)\right)\,dy, \qquad&\text{for}\,\, x \in \Omega \,\, \text{and}\,\, t>0,\\
u_t(x,t)\geq g(x,t), &\text{for}\,\, x \notin \Omega \,\, \text{and}\,\, t>0,\\
u(x,0)\geq u_0(x), &\text{for}\,\, x \in \Omega.
\end{cases}
\end{equation}
As usual, subsolutions are defined analogously by reversing the inequalities.
\end{defin}
\begin{lem}\label{com}
Let $u_0 \in C(\overline\Omega)$, $u_0 \geq 0$, and let $u \in C(\overline\Omega\times[0, T ])$ be a supersolution of
\eqref{pro2} with $g \geq 0$. Then, $u \geq 0$.
\end{lem}

\begin{proof}
Assume to the contrary that $u(x, t)$ is negative in some point. Let $v(x, t) = u(x, t) + \gamma t$ with $\gamma>0$ small such that $v$ is still negative somewhere. Then, if $(x_0, t_ 0 )$ is a point where $v$ attains its negative minimum, there it holds that $t_ 0 > 0$ and
\begin{align*}
v_t(x_0, t_ 0 )&=u_t(x_0, t_0 )+\gamma>\int_G K_\epsilon(x_0,y)\left(u(y,t_0)-u(x_0,t_0)\right)dy\\
                    &=\int_G K_\epsilon(x_0,y)\left(v(y,t_0)-v(x_0,t_0)\right)dy\geq0.
\end{align*}
 This contradicts that $(x_0 , t_ 0 )$ is a minimum of $v$. Thus, $u \geq 0$.
\end{proof}
Let $f(x,t)$ a function in $G\times (0,\infty)$, we consider next problem 
\begin{align}\label{pro2.0}
\left\{ \begin{array}{rclcc}
u^{\epsilon}_{t}(x,t) & = & \int_G K_\epsilon(x,y)\left(u(y,t)-u(x,t)\right)dy+f(x,t) & \mbox{ for } & (x,t)\in\Omega\times[0,T],\\
u^{\epsilon}(x,t) & = & g(x,t) & \mbox{ for } & x\notin\Omega, t\in[0,T],\\
u^{\epsilon}(x,0) & = & u_{0}(x) & \mbox{ for } & x\in\Omega,
\end{array}
\right.
\end{align}
\begin{cor}\label{comp}
Let $K_\epsilon\in L^\infty (G)$. Let $u_0$ and $v_0$ in $L^1(\Omega)$ with $u_0 \geq v_ 0$ and let the functions 
$g, h \in L^\infty  ((0, T ); L^1 (G\setminus\Omega))$ with $g \geq h$. Let $u$ be a solution of \eqref{pro2.0} with $u(x, 0) = u _0(x)$ and Dirichlet datum $g$, and let $v$ be a solution of \eqref{pro2.0} with $v(x, 0) = v_0(x)$
and datum $h$. Then, $u \geq v$ a.e. $\Omega$.
\end{cor}
\begin{proof}
Let $w = u - v$. Then, $w$ is a supersolution of \eqref{pro2} with initial datum $u_ 0 - v_ 0 \geq 0$ and datum $g - h \geq 0$. Using the continuity of the solutions with respect to the data and the fact that $K_\epsilon\in L^\infty (G)$, we may assume that $u, v \in C(\Omega\times [0, T ])$. By Lemma \eqref{com} we obtain that $w = u - v \geq 0$. So the corollary is proved.
\end{proof}
\begin{cor}\label{comp2}
Let $u,v \in C(\Omega\times [0, T ])$. If $u$ is a supersolution of \eqref{pro2.0} and $v$ is a subsolution of \eqref{pro2.0}, then $u \geq v$.
\end{cor}
\begin{proof}
It follows from the proof of the previous corollary.
\end{proof}

\section{Proof of the Main Theorems}\label{sec:Proofs}

\par The following Lemmas are the key for the proof of Theorems \ref{thm:MainMR} and \ref{thm:MainSYL}. To illustrate the technique we first prove a result which refers to the evolution problem stated in section \eqref{subsec:EvolEqn}.
\begin{lem}\label{lem:aproxSublap} Let $\Omega\subset G$ be a bounded domain, and 
let $v\in C^{2+\alpha}(G)$ for some $0<\alpha <1$. Then there exist constants $c$ and $c'$ that depends only of $v$, $J$ and $\Omega$ such that for all $\epsilon>0$  
\begin{align*} \left|\left|\mathcal{E}_{\epsilon}(v)-\frac{c'}{2}\mathcal{J}(v)\right|\right|_{L^{\infty}(\Omega)} \le & c \epsilon^{\alpha},
\end{align*}
where $\mathcal{J}(v)(x)=\sum\limits_{i=1}^{n_{1}}X_{i}^{2}v(x)$ denotes minus the subLaplacian.
\end{lem}
\begin{proof}
\par Let us begin by writing the formula that defines $\mathcal{E}_{\epsilon}$ by means of the global chart given by the fixed basis of the stratified Lie algebra $\mathfrak{g}$: for $x\in\Omega$, since 
\begin{align*}
\delta_{\epsilon}(y^{-1}) = & \exp(-\epsilon t_{1}X_{1} - \dots - \epsilon t_{n_{1}}X_{n_{1}}\\
& \qquad -\epsilon^{2}t_{n_{1}+1}X_{n_{1}+1}-\dots-\epsilon^{2}t_{n_{1}+n_{2}}X_{n_{1}+n_{2}}-\dots \\
& \qquad -\epsilon^{m}t_{n}X_{n}),
\end{align*}
for the coordinates $(t_{1},\dots,t_{n})\in\mathbb{R}^{n}$ adapted to the basis, we can write
\begin{align*}
\mathcal{E}_{\epsilon}v(x) = &  \frac{1}{\epsilon^{2}}\int\limits_{G}\left[v(x(\delta_{\epsilon}(y^{-1})))-v(x)\right]J(y)dy \\
 = & \frac{1}{\epsilon^{2}} \int\limits_{\mathbb{R}^{n}} \left( v\left( x\exp\left( -\epsilon t_{1}X_{1} -\dots-\epsilon^{m}t_{n}X_{n}\right) \right) -v(x) \right)\\
& \qquad \times J(\exp(t_{1}X_{1}+\dots+t_{n}X_{n})) dt_{1}\dots dt_{n}.
\end{align*}
\par Thus, from the Taylor expansion \eqref{simple.Taylor.expansion} discussed in section \ref{sec:Prelim},

\begin{align*}
 v\left( x\exp\left( -\epsilon t_{1}X_{1} -\dots-\epsilon^{m}t_{n}X_{n}\right) \right) -v(x)  = & -\epsilon \sum\limits_{i=1}^{n_{1}+n_{2}} t_{i}X_{i}v(x) \\ 
& + \frac{\epsilon^{2}}{2} \sum\limits_{i,j=1}^{n_{1}} t_{i}t_{j}X_{i}X_{j}v(x)+o(\left|\delta_\epsilon\left(t_{1}X_{1} +\dots+t_{n}X_{n}\right)\right|^{2}).
\end{align*}
Therefore,
\begin{align*}
\mathcal{E}_{\epsilon}v(x) = & \frac{1}{\epsilon^{2}} I + \frac{1}{\epsilon^{2}} II +\frac{1}{\epsilon^{2}} III,
\end{align*}
used \eqref{vani2}, we have
\begin{align*}
\left|\frac{1}{\epsilon^{2}} III\right|&=\left| \int_{\mathbb{R}^{n}}  \frac{o(\left|\delta_\epsilon\left(t_{1}X_{1} +\dots+t_{n}X_{n}\right)\right|^{2})}{\epsilon^2}J\left(\exp(t_{1}X_{1} +\dots+t_{n}X_{n})\right)dt_1\cdots dt_n\right|\\
&\leq c\int_{\mathbb{R}^{n}} \epsilon^\alpha\left|\exp(t_{1}X_{1} +\dots+t_{n}X_{n})\right|^{2+\alpha}\left|J\left(\exp(t_{1}X_{1} +\dots+t_{n}X_{n})\right)\right|dt_1\cdots dt_n\\
&\leq c\epsilon^\alpha.
\end{align*}

Where from properties \eqref{J.x}, \eqref{J.x^2} and \eqref{J.deltaij} we can compute
\begin{align*}
\frac{I}{\epsilon^2} = &  \frac{-1}{\epsilon^2}  \sum\limits_{i=1}^{n_{1}+n_{2}} \epsilon^{\lambda_j} X_{i}v(x)\int\limits_{\mathbb{R}^{n}} J(\exp(t_{1}X_{1}+\dots +t_{n}X_{n}))t_{i}dt_{1}\dots dt_{n} = 0, \\
\frac{II}{\epsilon^2} = & \frac{1}{2} \sum\limits_{i,j=1}^{n_{1}} X_{i}X_{j}v(x) \int\limits_{\mathbb{R}^{n}} J(\exp(t_{1}X_{1}+\dots +t_{n}X_{n})) t_{i}t_{j} dt_{1}\dots dt_{n} \\
 = & c' \frac{1}{2} \sum\limits_{i=1}^{n_{1}} X_{i}^{2}v(x).
\end{align*}
\par Finally, 
\begin{align*}
\left|\left|\mathcal{E}_{\epsilon}v(x)-\frac{c'}{2}\mathcal{J}v(x)\right|\right|_{L^{\infty}(\Omega)} \le & c\epsilon^\alpha.
\end{align*}
\end{proof}

\begin{lem}\label{lem:cuentas} Let $\Omega\subset G$ be a bounded domain, and let $v\in C^{2+\alpha}(G)$ for some $0<\alpha <1$. Then, there exists a constant $c$ that depends only of $v$, the matrix $A$, the vector $b$,  $J$ and $\Omega$ such that for all $\epsilon>0$  
\begin{align*} \left|\left|\mathcal{K}_{\epsilon}(v)-\mathcal{K}(v)\right|\right|_{L^{\infty}(\Omega)} \le & c \epsilon^{\alpha},
\end{align*}
where $\mathcal{K}$ is the operator defined in \eqref{operator}. 
\end{lem}

\begin{proof}
\par By changing variables via $z=\exp(L^{-1}(x)\exp^{-1}(\delta_{\epsilon^{-1}}(y^{-1}x)))$, since thus we have that $y=x\exp(\delta_{\epsilon}L(x)\exp^{-1}(z^{-1}))$ and $dy=\epsilon^{Q}\det(L(x))dz$ 
for $\epsilon>0$, the reescaled kernel operator becomes
\begin{align*}
\mathcal{K}_{\epsilon}v(x)=\frac{c(x)\det(L(x))}{\epsilon^{2}}\int\limits_{G}&a\left(\left(\exp\frac{M}{2}\delta_{\epsilon}W(x)(L^{t}(x))^{-1}\exp^{-1}z\right)^{-1}\right)J(z)\\
& (v(x\exp(\delta_{\epsilon}L(x)\exp^{-1}(z^{-1})))-v(x))dz,
\end{align*}
and by definition of the function $a$ it finally assumes the form
\begin{align*}
\mathcal{K}_{\epsilon}v(x)=\frac{2}{\epsilon^{2}C(J)M} \int\limits_{G}&\left(-\frac{M}{2}\sum\limits_{j=1}^{n}\epsilon^{\lambda_{j}}\tilde{b}_{j}(x)\sum\limits_{h=1}^{n}l_{hj}^{\ast}(x)\phi_{h}(\exp^{-1}z)+M\right) J(z)\\
& (v(x\exp(\delta_{\epsilon}L(x)\exp^{-1}(z^{-1})))-v(x))dz.
\end{align*}
\par Now let us write the formula in terms of the global chart as we did before (recall the proof of Lemma \ref{lem:aproxSublap}):
\begin{align*}
\mathcal{K}_{\epsilon}v(x)=\frac{2}{\epsilon^{2}C(J)M} \int\limits_{\mathbb{R}^{n}}&\left(-\frac{M}{2}\sum\limits_{j=1}^{n}\epsilon^{\lambda_{j}}\tilde{b}_{j}(x)\sum\limits_{h=1}^{n}l_{hj}^{\ast}(x)t_{h}+M\right) J\left(\exp\sum\limits_{r=1}^{n}t_{r}X_{r}\right)\\
& \left(v\left(x\exp\left(-\sum\limits_{i=1}^{n}\epsilon^{\lambda_{i}}\sum\limits_{k=1}^{n}l_{ik}(x)t_{k}X_{i}\right)\right)-v(x)\right)dt_{1}\dots dt_{n},
\end{align*}
where $t_h=\phi_{h}(\exp^{-1}z)$.

\par For the last factor we apply the Taylor expansion of homogeneous degree $2$ (recall formula \eqref{simple.Taylor.expansion})
\begin{align*}
& v\left(x\exp\left(-\sum\limits_{i=1}^{n}\epsilon^{\lambda_{i}}\sum\limits_{k=1}^{n}l_{ik}(x)t_{k}X_{i}\right)\right)-v(x) \\ 
& = -\sum\limits_{i=1}^{n_{1}+n_{2}}\epsilon^{\lambda_{i}}\sum\limits_{k=1}^{n} l_{ik}(x)t_{k} X_{i}v(x)  + \frac{\epsilon^{2}}{2} \sum\limits_{i,j=1}^{n_{1}}\left(\sum\limits_{k=1}^{n} l_{ik}(x)t_{k}\right)\left(\sum\limits_{h=1}^{n}l_{jh}(x)t_{h}\right)X_{i}X_{j}v(x) \\
& + o\left(\left|\sum\limits_{i=1}^{n}\epsilon^{\lambda_{i}}\sum\limits_{k=1}^{n}l_{ik}(x)t_{k}X_{i}\right|^2\right).
\end{align*}
Then we can split as follows
\begin{align*}
&\mathcal{K}_{\epsilon}v(x)=\mathcal{K}_{\epsilon,1}v(x)+\mathcal{K}_{\epsilon,2}v(x)+E_\epsilon. 
\end{align*}
By \eqref{vani2} 
\begin{align*}
 \left|E_\epsilon\right|&\leq \frac{1}{C(J)} \int\limits_{\mathbb{R}^{n}}\left|\left(-\frac{M}{2}\sum\limits_{j=1}^{n}\epsilon^{\lambda_{j}}\tilde{b}_{j}(x)\sum\limits_{h=1}^{n}l_{hj}^{\ast}(x)t_{h}+M\right) J\left(\exp\sum\limits_{r=1}^{n}t_{r}X_{r}\right)\right| \\
                        &\qquad\times\left|\frac{o\left(\left|\delta_\epsilon\left(\sum\limits_{i=1}^{n}\sum\limits_{k=1}^{n}l_{ik}(x)t_{k}X_{i}\right)\right|^2\right)}{\epsilon^2}\right|\, dt_1\cdots dt_n \\
                       &\leq  \frac{1}{C(J)} \int\limits_{\mathbb{R}^{n}}\left|\left(-\frac{M}{2}\sum\limits_{j=1}^{n}\epsilon^{\lambda_{j}}\tilde{b}_{j}(x)\sum\limits_{h=1}^{n}l_{hj}^{\ast}(x)t_{h}+M\right) J\left(\exp\sum\limits_{r=1}^{n}t_{r}X_{r}\right)\right| \\
                        &\qquad\times c \epsilon^\alpha \left|\exp\left( \sum\limits_{i=1}^{n}\sum\limits_{k=1}^{n}l_{ik}(x)t_{k}X_{i}\right)\right|^{2+\alpha} \, dt_1\cdots dt_n \\
												&= c \epsilon^\alpha.
\end{align*}

Now, for  $\mathcal{K}_{\epsilon,1}v(x)$ and $\mathcal{K}_{\epsilon,2}v(x)$,  by extensive use of properties \eqref{J.x}, \eqref{J.x^2} and \eqref{J.deltaij}, we have
\begin{align*}
\mathcal{K}_{\epsilon,1}v(x) = & \frac{1}{\epsilon^{2}C(J)}\sum\limits_{i=1}^{n_{1}+n_{2}}\epsilon^{\lambda_{i}}\sum\limits_{k=1}^{n}l_{ik}(x)\sum\limits_{j=1}^{n}\epsilon^{\lambda_{j}}\tilde{b}_{j}(x)\sum\limits_{h=1}^{n}l^{\ast}_{hj}(x) \\
& \left(\,\int\limits_{\mathbb{R}^{n}}J\left(\exp\sum\limits_{r=1}^{n}t_{r}X_{r}\right)t_kt_hdt_{1}\dots dt_{n}\right)X_{i}v(x)\\
= & \frac{1}{\epsilon^{2}}\sum\limits_{i=1}^{n_{1}+n_{2}}\epsilon^{\lambda_{i}}\sum\limits_{j=1}^{n}\epsilon^{\lambda_{j}}\tilde{b}_{j}(x)\sum\limits_{k=1}^{n}l_{ik}(x)l^{\ast}_{kj}(x)X_{i}v(x) \\
= & \frac{1}{\epsilon^{2}}\sum\limits_{i=1}^{n_{1}+n_{2}}\epsilon^{\lambda_{i}}\sum\limits_{j=1}^{n}\epsilon^{\lambda_{j}}\tilde{b}_{j}(x)\delta_{ij}X_{i}v(x)\\
= & \frac{1}{\epsilon^{2}}\sum\limits_{i=1}^{n_{1}+n_{2}}\epsilon^{2\lambda_{i}}\tilde{b}_{i}(x)X_{i}v(x) \\
= &\sum\limits_{i=1}^{n_{1}+n_{2}}b_{i}(x)X_{i}v(x);
\end{align*}
\begin{align*}
\mathcal{K}_{\epsilon,2}v(x) = & \frac{2}{\epsilon^{2}C(J)} \sum\limits_{i,j=1}^{n_{1}}\frac{\epsilon^{2}}{2}\left(\sum\limits_{k=1}^{n}l_{ik}(x)\sum\limits_{h=1}^{n}l_{jh}(x)\right) \\
& \left(\int\limits_{\mathbb{R}^{n}}J\left(\exp\sum\limits_{r=1}^{n}t_{r}X_{r}\right)t_kt_hdt_{1}\dots dt_{n}\right)X_{i}X_{j}v(x) \\
 = & \sum\limits_{i,j=1}^{n_{1}}a_{ij}(x)X_{i}X_{j}v(x).
\end{align*}
\par Thus the proof ends.
\end{proof}

\begin{lem} Let $\Omega\subset G$ be a bounded domain, and let $u\in C^{2+\alpha}(G)$ for some $0<\alpha <1$. Then, there exists a constant $c$ that depends only of $v$, the function $a$, $J$ and $\Omega$ such that for all $\epsilon>0$  
\begin{align*} ||\mathcal{L}_{\epsilon}(v)-\mathcal{L}(v)||_{L^{\infty}(\Omega)} \le c\epsilon^{\alpha},
\end{align*}
where $\mathcal{L}$ is the operator defined in \eqref{operator1}.
\end{lem}

\begin{proof}
\par Let us rewrite the operators as follows:
\begin{align*}
\mathcal{L}_{\epsilon}(v)(x)=&\frac{2C(J)}{\epsilon^{Q+2}} \int\limits_{G} J(\delta_{\epsilon^{-1}}y^{-1}x)a(y)[v(y)-v(x)]dy \\
& +\frac{2C(J)}{\epsilon^{Q+2}} \int\limits_{G} J(\delta_{\epsilon^{-1}}y^{-1}x)[a(y)-a(x)]v(x)dy.
\end{align*}

\par As usual, let us first change variables according to $z=\delta_{\epsilon^{-1}}y^{-1}x$, hence $y=\delta_{\epsilon}xz^{-1}$ and $dz=-\epsilon^{Q}dy$ and then write it in coordinates:
\begin{align*}
\mathcal{L}_{\epsilon}(v)(x) = &\frac{2C(J)}{\epsilon^{2}} \int\limits_{G} J(z)a(\delta_{\epsilon}xz^{-1})[v(\delta_{\epsilon}xz^{-1})-v(x)]dz \\
& +\frac{2C(J)}{\epsilon^{2}} \int\limits_{G} J(z)[a(\delta_{\epsilon}xz^{-1})-a(x)]v(x)dz \\
= & \frac{2C(J)}{\epsilon^{2}} \int\limits_{\mathbb{R}^{n}} J\left(\exp\sum\limits_{r=1}^{n}t_{r}X_{r}\right)a\left(x\exp\left(-\sum\limits_{k=1}^{n}\epsilon^{\lambda_{k}}t_{k}X_{k}\right)\right) \\
& \left[v\left(x\exp\left(-\sum\limits_{i=1}^{n}\epsilon^{\lambda_{i}}t_{i}X_{i}\right)\right)-v(x)\right]dt_{1}\dots dt_{n} \\
& +\frac{2C(J)}{\epsilon^{2}} \int\limits_{\mathbb{R}^{n}} J\left(\exp\sum\limits_{r=1}^{n}t_{r}X_{r}\right)\left[a\left(x\exp\left(-\sum\limits_{i=1}^{n}\epsilon^{\lambda_{i}}t_{i}X_{i}\right)\right)-a(x)\right] \\
 & \qquad \times v(x)dt_{1}\dots dt_{n}\\
 = & I + II,
\end{align*}
where $t_k=\phi_{k}(\exp^{-1}z)$.

The next step is to apply Taylor decomposition of homogeneous degree $2$ (recall formula \eqref{simple.Taylor.expansion}) to $v$ in $I$ and to $a$ in $II$:
\begin{align*}
 v\left(x\exp\left(-\sum\limits_{i=1}^{n}\epsilon^{\lambda_{i}}t_{i}X_{i}\right)\right)-v(x)  = & -\sum\limits_{i=1}^{n_{1}+n_{2}}\epsilon^{\lambda_{i}}t_{i} X_{i}v(x) \\ & + \frac{\epsilon^{2}}{2} \sum\limits_{i,j=1}^{n_{1}}t_{i}t_{j}X_{i}X_{j}v(x) + o\left(\left|\delta_\epsilon\left(\sum\limits_{i=1}^{n}t_{i}X_{i}\right)\right|^2\right),
\end{align*}
hence, by \eqref{vani2}
\begin{align*}
I = & \frac{2C(J)}{\epsilon^{2}} \int\limits_{\mathbb{R}^{n}} J\left(\exp\sum\limits_{r=1}^{n}t_{r}X_{r}\right)a\left(x\exp\left(-\sum\limits_{k=1}^{n}\epsilon^{\lambda_{k}}t_{k}X_{k}\right)\right) \\
& \qquad \times \left(-\sum\limits_{i=1}^{n_{1}+n_{2}}\epsilon^{\lambda_{i}}t_{i} X_{i}v(x) \right)dt_{1}\dots dt_{n} \\
& + C(J)\int\limits_{\mathbb{R}^{n}} J\left(\exp\sum\limits_{r=1}^{n}t_{r}X_{r}\right)a\left(x\exp\left(-\sum\limits_{k=1}^{n}\epsilon^{\lambda_{k}}t_{k}X_{k}\right)\right) \\
& \qquad \times \left( \sum\limits_{i,j=1}^{n_{1}}t_{i}t_{j}X_{i}X_{j}v(x)\right)dt_{1}\dots dt_{n} \\
& + \epsilon^{\alpha} c\int\limits_{\mathbb{R}^{n}} J\left(\exp\sum\limits_{r=1}^{n}t_{r}X_{r}\right)a\left(x\exp\left(-\sum\limits_{k=1}^{n}\epsilon^{\lambda_{k}}t_{k}X_{k}\right)\right)dt_{1}\dots dt_{n} \\
 = & I_{1}+I_{2}+\epsilon^\alpha c,
\end{align*}
and by applying Taylor formula again to $a$, but this time of homogeneous degree $1$, and extensive use of formulas \eqref{J.x}, \eqref{J.x^2} and \eqref{J.deltaij} it follows that
\begin{align*}
I_{1} = & \frac{2C(J)}{\epsilon^{2}} \int\limits_{\mathbb{R}^{n}} J\left(\exp\sum\limits_{r=1}^{n}t_{r}X_{r}\right)\left(a(x)-\sum\limits_{k=1}^{n_{1}}\epsilon t_{k}X_{k}(a)(x)+o\left(\left|\delta_\epsilon\left(\sum_{r=1}^{n}t_{r}X_{r}\right)\right|\right)\right) \\
& \times \left(-\sum\limits_{i=1}^{n_{1}+n_{2}}\epsilon^{\lambda_{i}}t_{i} X_{i}(v)(x) \right)dt_{1}\dots dt_{n} \\
= & \frac{2C(J)}{\epsilon^{2}} \sum\limits_{i=1}^{n_{1}+n_{2}}  \sum\limits_{k=1}^{n_{1}}\epsilon^{\lambda_{i}+1} X_{k}(a)(x) X_{i}(v)(x)\int\limits_{\mathbb{R}^{n}} J\left(\exp\sum\limits_{r=1}^{n}t_{r}X_{r}\right) t_{k}t_{i}  dt_{1}\dots dt_{n} \\
&+\frac{2C(J)}{\epsilon^{2}} \int\limits_{\mathbb{R}^{n}} J\left(\exp\sum\limits_{r=1}^{n}t_{r}X_{r}\right)o\left(\left|\delta_\epsilon\left(\sum_{r=1}^{n}t_{r}X_{r}\right)\right|\right) \left(-\sum\limits_{i=1}^{n_{1}+n_{2}}\epsilon^{\lambda_{i}}t_{i} X_{i}(v)(x) \right)dt_{1}\dots dt_{n} \\
= & 2\sum\limits_{i=1}^{n_{1}} X_{i}(a)(x) X_{i}(v)(x), \\
&+\frac{C(J)}{\epsilon} \int\limits_{\mathbb{R}^{n}} J\left(\exp\sum\limits_{r=1}^{n}t_{r}X_{r}\right)o\left(\left|\delta_\epsilon\left(\sum_{r=1}^{n}t_{r}X_{r}\right)\right|\right) \left(-\sum\limits_{i=1}^{n_{1}+n_{2}}\epsilon^{\lambda_{i}-1}t_{i} X_{i}(v)(x) \right)dt_{1}\dots dt_{n} \\
\end{align*}

Now, by \eqref{vani} and Theorem \ref{Thm:STI}, we get
\begin{align*}
&\left|\int\limits_{\mathbb{R}^{n}} J\left(\exp\sum\limits_{r=1}^{n}t_{r}X_{r}\right)\frac{o\left(\left|\delta_\epsilon\left(\sum_{r=1}^{n}t_{r}X_{r}\right)\right|\right)}{\epsilon} \left(-\sum\limits_{i=1}^{n_{1}+n_{2}}\epsilon^{\lambda_{i}-1}t_{i} X_{i}(v)(x) \right)dt_{1}\dots dt_{n}\right|\\
&\quad\leq \epsilon c\int\limits_{\mathbb{R}^{n}} \left|J\left(\exp\sum\limits_{r=1}^{n}t_{r}X_{r}\right)\left(-\sum\limits_{i=1}^{n_{1}+n_{2}}\epsilon^{\lambda_{i}-1}t_{i} X_{i}(v)(x) \right)\right|dt_{1}\dots dt_{n}\\
&\quad\leq \epsilon c.
\end{align*}

For $I_2$ we have
\begin{align*}
I_{2} = & C(J)\int\limits_{\mathbb{R}^{n}} J\left(\exp\sum\limits_{r=1}^{n}t_{r}X_{r}\right)\left(a(x)-\sum\limits_{k=1}^{n_{1}}\epsilon t_{k}X_{k}(a)(x)+o\left(\left|\delta_\epsilon\left(\sum_{r=1}^{n}t_{r}X_{r}\right)\right|\right)\right) \\
& \times \left( \sum\limits_{i,j=1}^{n_{1}}t_{i}t_{j}X_{i}X_{j}v(x)\right)dt_{1}\dots dt_{n} \\
= & C(J)\sum\limits_{j,i=1}^{n_{1}}a(x)X_{j}X_{i}(v)(x) \int\limits_{\mathbb{R}^{n}}  J\left(\exp\sum\limits_{rr=1}^{n}t_{r}X_{r}\right)t_{j}t_j dt_{1}\dots dt_{n}\\
  &\quad + C(J)\int\limits_{\mathbb{R}^{n}} J\left(\exp\sum\limits_{r=1}^{n}t_{r}X_{r}\right)o\left(\left|\delta_\epsilon\left(\sum_{r=1}^{n}t_{r}X_{r}\right)\right|\right)\left( \sum\limits_{i,j=1}^{n_{1}}t_{i}t_{j}X_{i}X_{j}v(x)\right)dt_{1}\dots dt_{n}  \\
= & \sum\limits_{i=1}^{n_{1}}a(x)X_{i}X_{i}(v)(x)\\
&\quad+ C(J)\int\limits_{\mathbb{R}^{n}} J\left(\exp\sum\limits_{r=1}^{n}t_{r}X_{r}\right)o\left(\left|\delta_\epsilon\left(\sum_{r=1}^{n}t_{r}X_{r}\right)\right|\right)\left( \sum\limits_{i,j=1}^{n_{1}}t_{i}t_{j}X_{i}X_{j}v(x)\right)dt_{1}\dots dt_{n} .
\end{align*}
Now, by \eqref{vani} and Theorem \ref{Thm:STI}, we get
\begin{align*}
\left|\int\limits_{\mathbb{R}^{n}} J\left(\exp\sum\limits_{r=1}^{n}t_{r}X_{r}\right)o\left(\left|\delta_\epsilon\left(\sum_{r=1}^{n}t_{r}X_{r}\right)\right|\right)\left( \sum\limits_{i,j=1}^{n_{1}}t_{i}t_{j}X_{i}X_{j}v(x)\right)dt_{1}\dots dt_{n} \right|\leq \epsilon C(J,v,a).
\end{align*}

Finally, by applying Taylor decomposition of homogeneous degree $2$ again to $a$,  
\begin{align*}
II = & \frac{2C(J)}{\epsilon^{2}} \int\limits_{\mathbb{R}^{n}} J\left(\exp\sum\limits_{r=1}^{n}t_{r}X_{r}\right) \\
 & \qquad \times \left[-\sum\limits_{i=1}^{n_{1}+n_{2}}\epsilon^{\lambda_{i}}t_{i} X_{i}(a)(x)  + \frac{\epsilon^{2}}{2} \sum\limits_{i,j=1}^{n_{1}}t_{i}t_{j}X_{i}X_{j}(a)(x) + o\left(\left|\delta_\epsilon\left(\sum\limits_{r=1}^{n}t_{r}X_{r}\right)\right|^2\right)\right] \\
& \qquad \times v(x)dt_{1}\dots dt_{n} \\
= & C(J) \sum\limits_{j,i=1}^{n_{1}} X_{j}X_{i}(a)(x)v(x) \int\limits_{\mathbb{R}^{n}} J\left(\exp\sum\limits_{r=1}^{n}t_{r}X_{r}\right)t_{i}t_j dt_{1}\dots dt_{n}\\
& \qquad +\frac{2C(J)}{\epsilon^{2}} \int\limits_{\mathbb{R}^{n}} J\left(\exp\sum\limits_{r=1}^{n}t_{r}X_{r}\right)o\left(\left|\delta_\epsilon\left(\sum\limits_{r=1}^{n}t_{r}X_{r}\right)\right|^2\right) v(x)dt_{1}\dots dt_{n} \\
= & \sum\limits_{i=1}^{n_{1}} X_{i}X_{i}(a)(x)v(x)\\
&+\frac{2C(J)}{\epsilon^{2}} v(x)\int\limits_{\mathbb{R}^{n}} J\left(\exp\sum\limits_{r=1}^{n}t_{r}X_{r}\right)o\left(\left|\delta_\epsilon\left(\sum\limits_{r=1}^{n}t_{r}X_{r}\right)\right|^2\right) dt_{1}\dots dt_{n}, \\
\end{align*}
by \eqref{vani2}
\begin{align*}
&\left|\frac{2C(J)}{\epsilon^{2}} v(x)\int\limits_{\mathbb{R}^{n}} J\left(\exp\sum\limits_{r=1}^{n}t_{r}X_{r}\right)o\left(\left|\delta_\epsilon\left(\sum\limits_{r=1}^{n}t_{r}X_{r}\right)\right|^2\right) dt_{1}\dots dt_{n}\right| \\
&\qquad\leq \|v\|_{L^\infty(\Omega)} c\epsilon.
\end{align*}
\end{proof}

\par Next we turn to the proof of Theorem \ref{thm:MainMR}. The proof of Theorem \ref{thm:MainSYL} follows the same lines.
\begin{proof}[Proof of Theorem \ref{thm:MainMR}]
\par Let $v(\cdot,t)\in C^{2+\alpha}(\Omega)$ be a solution of problem \eqref{0.2}, and define an extension $\tilde{v}$ of $v$ to the space $C^{2+\alpha,1+\alpha}(G\times [0,T])$ such that \begin{equation}\label{5.1}%
\begin{cases}
\tilde{v}_{t}(x,t)=\mathcal{K}(\tilde{v})(x,t),\qquad & x\in \Omega,\quad t>0,\\
\tilde{v}(x,t)=\tilde{g}(x,t), & x\notin \Omega,\quad t>0,\\
\tilde{v}(x,0)=u_{0}(x), & x\in\Omega,
\end{cases}
\end{equation}
where $\tilde{g}$ is a smooth function which satisfies $\tilde{g}(x,t)=g(x,t)$ if $x\in\partial\Omega$ and $\tilde{g}(x,t)=g(x,t)+O(\epsilon)$ if $x\approx\partial\Omega$, (in the sense $\lim_{\epsilon\to 0}O(\epsilon)=0$).
\par Let us define now the difference $w^{\epsilon}(x,t)=\tilde{v}(x,t)-u^{\epsilon}(x,t)$. Thus defined, $w^{\epsilon}$ satisfies
\begin{equation}\label{5.2}%
\begin{cases}
w^{\epsilon}_{t}(x,t)=\mathcal{K}(\tilde{v})(x,t)-\mathcal{K}_{\epsilon}\tilde{v}(x,t)+\mathcal{K}_{\epsilon}w^{\epsilon}(x,t),\qquad & x\in \Omega,\quad t>0,\\
w^{\epsilon}(x,t)=g(x,t)-\tilde{g}(x,t), & x\notin \Omega,\quad t>0,\\
w^{\epsilon}(x,0)=0, & x\in\Omega.
\end{cases}
\end{equation}
From Lemma \ref{lem:cuentas} exists a constant $K_1$ dependent only of $\tilde v$ and the differential operator $\mathcal{K}$ such that for all $\epsilon>0$
\begin{align*}
\left|\mathcal{K}\tilde{v}(x,t)-\mathcal{K}_{\epsilon}\tilde{v}(x,t)\right|\leq K_1\epsilon^\alpha.
\end{align*}

Let $\overline{w}(x,t)=K_{1}\epsilon^\alpha t+K_{2}\epsilon$, where $K_2>0$ is a constant independent of $\epsilon$ to be chosen later. Now we see that $\overline{w}(x,t)$ is a supersolution of the problem  \eqref{5.2}. Since $\overline{w}(x,t)$ does not depend on $x$, we have 
\begin{align*}
 \mathcal{K}_{\epsilon}\overline{w}(x,t)=\int_{G}{K}_{\epsilon}(x,y)(\overline{w}(y,t)-\overline{w}(x,t))\,dy=0, 
\end{align*}
and follows that 
\begin{align*}
\overline{w}_{t}(x,t)=K_{1}\epsilon^\alpha \ge\mathcal{K}\tilde{v}(x,t)-\mathcal{K}_{\epsilon}\tilde{v}(x,t)+\mathcal{K}_{\epsilon}\overline{w}(x,t).
\end{align*}


Also, $\overline{w}(x,0)>0$ and by the definition of $\tilde{g}$, we can choose $K_{2}>0$ such that
\begin{align*}
\overline{w}(x,t)\ge K_{2}\epsilon \ge O(\epsilon),
\end{align*}
for $x\in\Omega^{c}$, $x\approx\partial\Omega$, $t>0$. Hence $\overline{w}$ is indeed a supersolution of \eqref{5.2}.

\par From the comparison principle (Corollary \ref{comp2}) we get that $\tilde{v}-u^{\epsilon}\le \overline{w}(x,t)=K_{1}\epsilon^\alpha t+K_{2}\epsilon$.
\par Applying the same arguments for $\underline{w}(x,t)=-\overline{w}(x,t)$ we obtain that $\underline{w}(x,t)$ is a subsolution of problem \eqref{5.2} and again by the comparison principle, 
\begin{align*}
-K_{1}\epsilon^\alpha t-K_{2}\epsilon\le\tilde{v}-u^{\epsilon}\le K_{1}\epsilon^\alpha t+K_{2}\epsilon.
\end{align*}
Therefore,
\begin{align*}
||\tilde{v}-u^{\epsilon}||_{L^{\infty}(\Omega\times[0,T])} \le K_{1}\epsilon^\alpha T+K_{2}\epsilon \to 0.
\end{align*}
\end{proof}


%
%


\begin{thebibliography}{}
%
%



\bibitem{BV}
 M. Bodnar, J. J. L. Velazquez, An integro--differential equation arising as a limit of individual cell-based models, J. Differential Equations, 222, 341--380  (2006).

\bibitem{BLU}
 A. Bonfiglioli, E. Lanconelli, F. Uguzzoni, Uniform Gaussian estimates for the fundamental solutions for heat operators on Carnot groups, Adv. Differential Equations 7, 10, 1153--1192  (2002).

\bibitem{BBLU}
 M. Bramanti, L. Brandolini, E. Lanconelli, F. Uguzzoni, Non--divergence equations structured on H\"ormander vector fields: heat kernels and Harnack inequalities, Memoirs of the AMS 204, 961,  1--136 (2010).

\bibitem{BF}
 M. Bramanti, M. S. Fanciullo, $C^{k,\alpha}$--regularity of solutions to quasilinear equations structured on H\"{o}rmander's vector fields, Nonlinear Analysis 92, DOI: 10.1016/j.na.2013.06.01 (2013).

\bibitem{CG}
 L. Capogna, N. Garofalo, Regularity of minimizers of the calculus of variations in Carnot groups via hypoellipticity of systems of H\"{o}rmander type,  Journal of the European Mathematical Society, 5(1), 1--40  (2003).

\bibitem{CF}
 C. Carrillo, P. Fife, Spatial effects in discrete generation population models. J. Math. Biol., 50(2), 161--188  (2005).

\bibitem{CCR}
 E. Chasseigne, M. Chaves, J. D. Rossi, Asymptotic behavior for nonlocal diffusion equations, Journal de math\'ematiques pures et appliqu\'ees, 86(3), 271--291  (2006).

\bibitem{CER}
 C. Cortazar, M. Elgueta, J. D. Rossi, Nonlocal diffusion problems that approximate the heat equation with Dirichlet boundary conditions, Israel Journal of Mathematics, 170(1), 53--60  (2009).

\bibitem{DG}
 D. Danielli, N. Garofalo, Interior Cauchy--Schauder estimates for the heat flow in Carnot Caratheodory spaces, Methods Appl. Anal. 15(1), 121--136 (2008).

\bibitem{D}
 J. L. Dyer, A nilpotent Lie algebra with nilpotent automorphism group, Proc. Symp. Pure Math., 4, 33--49  (1961).

\bibitem{Fi}
 P. Fife, Some Nonclassical Trends in Parabolic and Parabolic-like Evolutions, Trends in Nonlinear Analysis, Springer, Berlin, Heidelberg (2003).

\bibitem{FS}
 G. B. Folland, E. M. Stein, Hardy Spaces on Homogeneous Groups, Princeton University Press, (1982).

\bibitem{FL}
 N. Fournier, P. Laurencot, Well--posedness of Smoluchowski's coagulation equation for a class of homogeneous kernels, J. Funct. Anal., 233, 351--379  (2006).

\bibitem{H}
 L. H\"ormander, Hypoelliptic second order differential equations, Acta Math. 119, 147--171 (1967).

\bibitem{KOJ}
 S. Kindermann, S. Osher, P. W. Jones, Deblurring and denoising of images by nonlocal functionals, Multiscale Model. Simul., 4, 1091--1115  (2005).

\bibitem{ML}
 A. Mogilner, L. Edelstein-Keshet, A non--local model for a swarm, J. Math. Biol., 38, 534--570 (1999).

\bibitem{MR}
 A. Molino, J.D. Rossi, Nonlocal diffusion problems that approximate a parabolic equation with spatial dependence,  Z. Angew. Math. Phys., 67--41 (2016).

\bibitem{R}
 F. Rossi, Large time behaviour for the heat equation on Carnot groups, Nonlinear Differ. Equ. Appl. 20(3), 1393--1407 (2013).

\bibitem{SLY}
 J. W. Sun, W. T. Li, F. I. Yang, Approximate the Fokker--Planck equation by a class of nonlocal dispersal problems,  Nonlinear Analysis: Theory, Methods and Applications, Vol 74, 3501--3509 (2011).

\bibitem{V}
V. S. Varadarajan, Lie Groups, Lie Algebras, and Their Representations, Springer-Verlag New York, (1984).

\bibitem{Vi}
 R. E Vidal, Nonlocal heat equations in Heisenbreg group, Nonlinear Differential Equations and Applications, 24--57  (2017).

\bibitem{WLL}
 J. Wang, Q. Liao, D. Liao, et al, $C^1-$partial regularity for sub-elliptic systems with Dini continuous coefficients in Carnot groups, Nonlinear Anal.TMA, 169 242--264 (2018).

\bibitem{XZ}
 C. Xu, C. Zuily, Higher interior regularity for quasilinear sub-elliptic systems, Calc. Var. Partial Differential Equations, 5, 323--343 (1997).

\end{thebibliography}


\end{document}